\documentclass[12pt]{article}

\usepackage[margin=1in]{geometry}

\usepackage{amsmath}
\usepackage{amssymb}
\usepackage{amsthm}
\usepackage{amsfonts}
\usepackage{enumerate}
\usepackage{graphicx}
\usepackage{url}
\usepackage[table]{xcolor}
\usepackage{soul}
\usepackage{enumitem}

\usepackage{subcaption}

\usepackage{ytableau}

\usepackage{tikz}
\usetikzlibrary{arrows,shapes,positioning,decorations.markings,decorations.pathreplacing}
\tikzset{->-/.style={decoration={
  markings,
  mark=at position .5 with {\arrow{>}}},postaction={decorate}}}

\usepackage{array}
\newcolumntype{?}{!{\vrule width 3pt}}

\newtheorem{theorem}{Theorem}[section]

\newtheorem{corollary}[theorem]{Corollary}
\newtheorem{proposition}[theorem]{Proposition}

\DeclareMathOperator{\im}{\text{Im}}

\begin{document}

\title{\bf Generalized Path Pairs \& \\Fuss-Catalan Triangles}
\author{Paul Drube\\
\small Department of Mathematics and Statistics\\[-0.8ex]
\small Valparaiso University\\[-0.8ex] 
\small Valparaiso, Indiana, U.S.A.\\
\small\tt paul.drube@valpo.edu\\
}

\maketitle

\begin{abstract}

Path pairs are a modification of parallelogram polyominoes that provide yet another combinatorial interpretation of the Catalan numbers.  More generally, the number of path pairs of length $n$ and distance $\delta$ corresponds to the $(n-1,\delta-1)$ entry of Shapiro's so-called Catalan triangle.  In this paper, we widen the notion of path pairs $(\gamma_1,\gamma_2)$ to the situation where $\gamma_1$ and $\gamma_2$ may have different lengths, and then enforce divisibility conditions on runs of vertical steps in $\gamma_2$.  This creates a two-parameter family of integer triangles that generalize the Catalan triangle and qualify as proper Riordan arrays for many choices of parameters.  In particular, we use generalized path pairs to provide a new combinatorial interpretation for all entries in every proper Riordan array $\mathcal{R}(d(t),h(t))$ of the form $d(t) = C_k(t)^i$, $h(t) = t \kern+1pt C_k(t)^k$, where $1 \leq i \leq k$ and $C_k(t)$ is the generating function for some sequence of Fuss-Catalan numbers (some $k \geq 2$).  Closed formulas are then provided for the number of generalized path pairs across an even broader range of parameters, as well as for the number of ``weak" path pairs with a fixed number of non-initial intersections.
\end{abstract}

\section{Introduction}
\label{sec: intro}

The Catalan numbers are a seemingly ubiquitous sequence of positive integers whose $n^{th}$ entry is $C_n = \frac{1}{n+1} \binom{2n}{n}$.  The Catalan numbers satisfy the recurrence $C_{n+1} = \sum_{i+j = n} C_{i} C_{j}$ for all $n \geq 0$, which translates to the ordinary generating function $C(t) = \sum_{n=0}^\infty C_n t^n$ as the relation $C(t) = t \kern+1pt C(t)^2 + 1$.  It follows that $C(t) = \frac{1-\sqrt{1-4t}}{2t}$.

Hundreds of combinatorial interpretations for the Catalan numbers have been compiled by Stanley \cite{Stanley}.  One such interpretation identifies $C_n$ with the number of parallelogram polyominoes with semiperimeter $n+1$.  These are ordered pairs of lattice paths $(\gamma_1,\gamma_2)$ that satisfy all of the following:

\begin{enumerate}
\item Both $\gamma_1$ and $\gamma_2$ are composed of $n+1$ steps from the step set $\lbrace E=(1,0),N=(0,1) \rbrace$, where $\gamma_1$ must begin with an $N$ step and $\gamma_2$ must begin with an $E$ step,
\item Both $\gamma_1$ and $\gamma_2$ begin at $(0,0)$ and end at the same point, and
\item $\gamma_1$ and $\gamma_2$ only intersect at their initial and final points.

\end{enumerate}

See Figure \ref{fig: intro polyomino example} for an illustration of all parallelogram polyominoes with semiperimeter $4$, noting that the number of such paths is $C_3 = 5$.

\begin{figure}[ht!]
\centering
\begin{tikzpicture}
[scale=.5,auto=left,every node/.style={circle, fill=black,inner sep=1.1pt}]
	\node (0*) at (0,0) {};
	\node (a1*) at (0,1) {};
	\node (a2*) at (0,2) {};
	\node (a3*) at (0,3) {};
	\node (a4*) at (1,3) {};
	\node (b1*) at (1,0) {};
	\node (b2*) at (1,1) {};
	\node (b3*) at (1,2) {};
	\node (b4*) at (1,3) {};
	\draw[very thick] (0*) to (a1*);	
	\draw[very thick] (a1*) to (a2*);
	\draw[very thick] (a2*) to (a3*);
	\draw[dotted] (a3*) to (a4*);
	\draw[very thick] (0*) to (b1*);
	\draw[very thick] (b1*) to (b2*);
	\draw[very thick] (b2*) to (b3*);
	\draw[dotted] (b3*) to (b4*);
\end{tikzpicture}
\hspace{.15in}
\begin{tikzpicture}
[scale=.5,auto=left,every node/.style={circle, fill=black,inner sep=1.1pt}]
	\node (0*) at (0,0) {};
	\node (a1*) at (0,1) {};
	\node (a2*) at (1,1) {};
	\node (a3*) at (2,1) {};
	\node (a4*) at (3,1) {};
	\node (b1*) at (1,0) {};
	\node (b2*) at (2,0) {};
	\node (b3*) at (3,0) {};
	\node (b4*) at (3,1) {};
	\draw[very thick] (0*) to (a1*);	
	\draw[very thick] (a1*) to (a2*);
	\draw[very thick] (a2*) to (a3*);
	\draw[dotted] (a3*) to (a4*);
	\draw[very thick] (0*) to (b1*);
	\draw[very thick] (b1*) to (b2*);
	\draw[very thick] (b2*) to (b3*);
	\draw[dotted] (b3*) to (b4*);
\end{tikzpicture}
\hspace{.15in}
\begin{tikzpicture}
[scale=.5,auto=left,every node/.style={circle, fill=black,inner sep=1.1pt}]
	\node (0*) at (0,0) {};
	\node (a1*) at (0,1) {};
	\node (a2*) at (0,2) {};
	\node (a3*) at (1,2) {};
	\node (a4*) at (2,2) {};
	\node (b1*) at (1,0) {};
	\node (b2*) at (1,1) {};
	\node (b3*) at (2,1) {};
	\node (b4*) at (2,2) {};
	\draw[very thick] (0*) to (a1*);	
	\draw[very thick] (a1*) to (a2*);
	\draw[very thick] (a2*) to (a3*);
	\draw[dotted] (a3*) to (a4*);
	\draw[very thick] (0*) to (b1*);
	\draw[very thick] (b1*) to (b2*);
	\draw[very thick] (b2*) to (b3*);
	\draw[dotted] (b3*) to (b4*);
\end{tikzpicture}
\hspace{.15in}
\begin{tikzpicture}
[scale=.5,auto=left,every node/.style={circle, fill=black,inner sep=1.1pt}]
	\node (0*) at (0,0) {};
	\node (a1*) at (0,1) {};
	\node (a2*) at (1,1) {};
	\node (a3*) at (1,2) {};
	\node (a4*) at (2,2) {};
	\node (b1*) at (1,0) {};
	\node (b2*) at (2,0) {};
	\node (b3*) at (2,1) {};
	\node (b4*) at (2,2) {};
	\draw[very thick] (0*) to (a1*);	
	\draw[very thick] (a1*) to (a2*);
	\draw[very thick] (a2*) to (a3*);
	\draw[dotted] (a3*) to (a4*);
	\draw[very thick] (0*) to (b1*);
	\draw[very thick] (b1*) to (b2*);
	\draw[very thick] (b2*) to (b3*);
	\draw[dotted] (b3*) to (b4*);
\end{tikzpicture}
\hspace{.15in}
\begin{tikzpicture}
[scale=.5,auto=left,every node/.style={circle, fill=black,inner sep=1.1pt}]
	\node (0*) at (0,0) {};
	\node (a1*) at (0,1) {};
	\node (a2*) at (0,2) {};
	\node (a3*) at (1,2) {};
	\node (a4*) at (2,2) {};
	\node (b1*) at (1,0) {};
	\node (b2*) at (2,0) {};
	\node (b3*) at (2,1) {};
	\node (b4*) at (2,2) {};
	\draw[very thick] (0*) to (a1*);	
	\draw[very thick] (a1*) to (a2*);
	\draw[very thick] (a2*) to (a3*);
	\draw[dotted] (a3*) to (a4*);
	\draw[very thick] (0*) to (b1*);
	\draw[very thick] (b1*) to (b2*);
	\draw[very thick] (b2*) to (b3*);
	\draw[dotted] (b3*) to (b4*);
\end{tikzpicture}

\caption{The $C_3 = 5$ parallelogram polyominoes with semiperimeter $4$, with the corresponding path pairs of length $3$ (and $\delta = 1$) appearing as the bolded edges.}
\label{fig: intro polyomino example}
\end{figure}
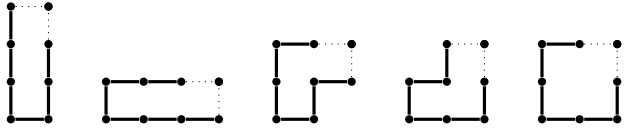

Generalizing the notion of parallelogram polyominoes are (fat) path pairs, as introduced by Shapiro \cite{Shapiro} and developed by Deutsch and Shapiro \cite{DS}.  A \textbf{path pair of length $\mathbf{n}$} is an ordered pair $(\gamma_1,\gamma_2)$ of lattice paths that satisfy all of the following:

\begin{enumerate}[itemsep=2pt]
\item Both $\gamma_1$ and $\gamma_2$ are composed of $n$ steps from the step set $\lbrace E=(1,0), N=(0,1) \rbrace$,
\item Both $\gamma_1$ and $\gamma_2$ begin at $(0,0)$, and
\item Apart from at $(0,0)$, $\gamma_1$ stays strongly above $\gamma_2$.
\end{enumerate}

Now consider the path pair $(\gamma_1,\gamma_2)$, and suppose that $\gamma_1$ terminates at $(x_1,y_1)$ while $\gamma_2$ terminates at $(x_2,y_2)$.  Clearly $x_1 < x_2$ and $y_1 > y_2$.  The path pair $(\gamma_1,\gamma_2)$ is said to have \textbf{distance} $\delta$ if $x_2 - x_1 = \delta$, and in this case we write $\vert \gamma_2 - \gamma_1 \vert = \delta$.  We henceforth use $\mathcal{P}_{n,\delta}$ to denote the set of all path pairs of length $n$ and distance $\delta$.

There is a simple bijection between $\mathcal{P}_{n,1}$ and parallelogram polynomials of semiperimeter $n+1$, via a map that adds an $E$ step to the end of $\gamma_1$ and a $N$ step to the end of $\gamma_2$.  See Figure \ref{fig: intro polyomino example} for an illustration of the $n=3$ case.  It follows that $\mathcal{P}_{n,1} = C_n$ for all $n \geq 0$.

Enumeration of $\mathcal{P}_{n,\delta}$ for all $\delta \geq 1$ and $n \geq 1$ was addressed by Shapiro \cite{Shapiro}, who identified $\vert \mathcal{P}_{n,\delta} \vert = \frac{2\delta}{2n} \binom{2n}{n-\delta}$ with the $(n-1,\delta-1)$ entry of his so-called Catalan triangle.  See Figure \ref{fig: intro Shapiro's Catalan triangle} for the first five rows of Shapiro's Catalan triangle, an infinite lower-triangular matrix (with zero entries suppressed) whose entries $d_{i,j}$ are generated by the recurrence $d_{0,0} = 1$ and $d_{i,j} = d_{i-1,j-1} + 2 d_{i-1,j} + d_{i-1,j+1}$ for all $i \geq 1, 0\leq j \leq i$.\footnote{Shapiro's Catalan triangle should not be confused with the ``Catalan triangle" whose $(i,j)$ entry is the ballot number $d_{i,j} = \frac{j+1}{i+1} \binom{2i-j}{i}$.  We alternatively refer to this second infinite lower-triangular matrix as the ballot triangle.  See Aigner \cite{Aigner} for connections between the ballot triangle and the Catalan triangle.}

\begin{figure}[ht!]
\centering
\setlength{\tabcolsep}{4pt}
\begin{tabular}{l l l l l}
1 & & & & \\
2 & 1 & & & \\
5 & 4 & 1 & & \\
14 & 14 & 6 & 1 & \\
42 & 48 & 27 & 8 & 1 \\
\end{tabular}
\caption{The first five rows of Shapiro's Catalan triangle.}
\label{fig: intro Shapiro's Catalan triangle}
\end{figure}

The Catalan triangle is a well-known example of a proper Riordan array.  Given a pair of generating functions $d(t)$ and $h(t)$ such that $d(0) \neq 0$, $h(0) = 0$, and $h'(0) \neq 0$, the associated proper Riordan array $\mathcal{R}(d(t),h(t))$ is the infinite lower-triangular matrix whose $(i,j)$ entry is $d_{i,j} = [t^i]d(t)h(t)^j$.  Here we use the standard notation in which $[t^i]$ identifies the coefficient of $t^i$ in a power series.  It may be verified that Shapiro's Catalan triangle is the proper Riordan array with $d(t) = C(t)^2$ and $h(t) = t \kern+2pt C(t)^2$.

For general information about Riordan arrays, see Rogers \cite{Rogers} or Merlini, Rogers, Sprugnoli and Verri \cite{MRSV}.  For a more focused discussion about how Riordan arrays similar to the Catalan triangle may be used to define so-called ``Catalan-like numbers", see Aigner \cite{Aigner2}.

Central to our work is the fact that every proper Riordan array $\mathcal{R}(d(t),h(t))$ possesses sequences of integers $\lbrace z_i \rbrace_{i=0}^\infty$ and $\lbrace a_i \rbrace_{i=0}^\infty$ such that 

\begin{equation}
\label{eq: Riordan array AZ recurrence}
d_{n,k} =
\begin{cases}
z_0 \kern+1pt d_{n-1,k} + z_1 \kern+1pt d_{n-1,k+1} + z_2 \kern+1pt d_{n-1,k+2} + \hdots & \text{for } k=0 \text{ and all } n \geq 1; \\[4pt]
a_0 \kern+1pt d_{n-1,k-1} + a_1 \kern+1pt d_{n-1,k} + a_2 \kern+1pt d_{n-1,k+1} + \hdots & \text{for all } k \geq 1 \text{ and } n \geq 1.
\end{cases}
\end{equation}

These sequences are referred to as the $Z$-sequence and the $A$-sequence of $\mathcal{R}(d(t),h(t))$, respectively.  When represented as the generating functions $Z(t) = \sum_i z_i \kern+1pt t^i$ and $A(t) = \sum_i a_i \kern+1pt t^i$, the $Z$- and $A$-sequences of a proper Riordan array are known to satisfy the relations

\begin{equation}
\label{eq: Riordan arrays AZ to dh}
d(t) = \frac{d(0)}{1-t \kern+1pt Z(h(t))} \hspace{.4in} h(t) = t \kern+1pt A(h(t))
\end{equation}

The defining recurrence of the Catalan triangle implies that it is a proper Riordan array with $Z(t) = 2 + t$ and $A(t) = 1 + 2t = t^2 = (1+t)^2$.  %One may use \eqref{eq: Riordan arrays AZ to dh} to verify that these sequences correspond to $d(t) = C(t)^2$ and $h(t) = t \kern+1pt C(t)^2$ requires the Catalan identity $C(t) = t \kern+1pt C(t)^2 + 1$.

As they play a major role in what follows, we pause to recap a few facts about the one-parameter Fuss-Catalan numbers, henceforth referred to as the $k$-Catalan numbers.  For any $k \geq 2$, the $k$-Catalan numbers are an integer sequence whose $n^{th}$ entry is $C_n^k = \frac{1}{kn+1} \binom{kn+1}{n}$.  Observe that the $k=2$ case corresponds to the ``original" Catalan numbers.  For any $k \geq 2$, the $k$-Catalan numbers satisfy the recurrence $C_{n+1}^k = \sum_{i_1 + \hdots + i_k} \kern-2pt C^k_{i_1} \kern-2pt \hdots \kern-1pt C^k_{i_k}$ for all $n \geq 0$, implying that their generating functions $C_k(t) = \sum_{n=0}^\infty C_n^k t^n$  satisfy $C_k(t) = t \kern+1pt C_k(t)^k + 1$.  For an introduction to the $k$-Catalan numbers, see Hilton and Pederson \cite{HP}.  For a list of combinatorial interpretations for the $k$-Catalan numbers, see Heubach, Li and Mansour \cite{HLM}.

\subsection{Outline of Results}
\label{subsec: outline}

The goal of this paper is to simultaneously explore several generalizations of path pairs.  Firstly, we eliminate the requirement that the two paths of $(\gamma_1,\gamma_2)$ have equal length, setting $\epsilon = \vert \gamma_2 \vert - \vert \gamma_1 \vert$ and examining the full range of differences $\epsilon \geq 0$ with $\vert \gamma_1 \vert \geq 0$.  We also enforce conditions on the $N$ steps of $\gamma_2$ that are designed to mirror the generalization of the Catalan numbers to the $k$-Catalan numbers.  We refer to the resulting combinatorial objects as $k$-path pairs of length $(n-\epsilon,n)$.

Section \ref{sec: k-path pairs} focuses upon the enumeration of $k$-path pairs.  In Subection \ref{subsec: k-path pairs, case 1}, we construct a two-parameter collection of infinite lower-triangular arrays $A^{k,\epsilon}$, whose entries correspond to the number of $k$-path pairs of varying lengths and distances.  For all $0 \leq \epsilon \leq k-1$, Theorem \ref{thm: k-path pairs as Riordan arrays} identifies the triangle $A^{k,\epsilon}$ with the proper Riordan array $\mathcal{R}(d(t),h(t))$ where $d(t) = C_k(t)^{k-\epsilon}$ and $h(t) = t \kern+1pt C_k(t)^k$.  In Subsection \ref{subsec: k-path pairs, case 2}, we directly enumerate sets of $k$-path pairs for all $k \geq 2$ and $\epsilon \leq 0$.  Theorem \ref{thm: summation decomposition corollary} uses the results of Subsection \ref{subsec: k-path pairs, case 2} to derive a closed formula for the size of all such sets, and Theorem \ref{thm: summation decomposition corollary 2} provides a significantly simplified formula within the range of $0 \leq \epsilon \leq (k-1)\delta$.

Section \ref{sec: weak path pairs} introduces a related generalization where we now allow the two paths $(\gamma_1,\gamma_2)$ to intersect away from $(0,0)$, so long as $\gamma_1$ stays weakly above $\gamma_2$ for the entirety of its length.  Theorem \ref{thm: weak path pair enumeration} applies the techniques of Section \ref{sec: k-path pairs} to derive a closed formula for the number of ``weak $k$-path pairs" whose paths intersect precisely $m$ times away from $(0,0)$, assuming that we restrict ourselves to the range $0 \leq \epsilon \leq (k-1)\delta$.

\section{Generalized $k$-Path Pairs}
\label{sec: k-path pairs}

Take any pair of integers $n,\epsilon$ such that $0 \leq \epsilon < n$.  Then define $\mathcal{P}^\epsilon_{n,\delta}$ to be the collection of ordered pairs $(\gamma_1,\gamma_2)$ of lattice paths that satisfy all of the following:

\begin{enumerate}[itemsep=2pt]
\item Both $\gamma_1$ and $\gamma_2$ begin at $(0,0)$ and use steps from $\lbrace E=(1,0), N=(0,1) \rbrace$,
\item $\gamma_2$ is composed of precisely $n$ steps, the first of which is an $E$ step,
\item $\gamma_1$ is composed of precisely $n-\epsilon$ steps, the first of which is a $N$ step,
\item $\gamma_1$ and $\gamma_2$ do not intersect apart from at $(0,0)$, and
\item The difference between the terminal $x$ coordinates of $\gamma_1$ and $\gamma_2$ is $\delta$.
\end{enumerate}

The case $\epsilon=0$ obviously corresponds to the original notion of path pairs.  If $\gamma_2$ terminates at $(x_2,y_2)$, then $\gamma_1$ terminates at $(x_1,y_1) = (x_2 - \delta,y_2 + \delta - \epsilon)$.  In particular, $y_1 - y_2 \geq 0$ precisely when $\delta \geq \epsilon$. 

Now fix $k \geq 2$, and consider some $(\gamma_1,\gamma_2) \in \mathcal{P}_{n,\delta}^\epsilon$.  The path pair $(\gamma_1,\gamma_2)$ is said to be a \textbf{$\mathbf{k}$-path pair of length $\mathbf{(n-\epsilon,n)}$ and distance $\mathbf{\delta}$} if the bottom path $\gamma_2 = E^1 N^{b_1} E^1 N^{b_2} \hdots E^1 N^{b_m}$ satisfies $b_i = (k \kern-1pt - \kern-1pt 2) \kern-4pt \mod \kern-3pt(k \kern-1pt - \kern-1pt 1)$ for all $i$.  Clearly, $2$-path pairs correspond to the notion of path pairs discussed above.

For any $k$-path pair $(\gamma_1,\gamma_2)$, the bottom path $\gamma_2$ must decompose into a sequence of length-$(k \kern-1pt - \kern-1pt 1)$ subpaths, each of which is either $N^{k-1}$ or $E^1 N^{k-2}$.  In particular, the length $n$ of $\gamma_2$ must be divisible by $k \kern-1pt - \kern-1pt 1$.  To avoid a large number of empty sets, we define $\mathcal{P}_{n,\delta}^{k,\epsilon}$ to be the collection of all $k$-path pairs of length $((k \kern-1pt - \kern-1pt 1)n-\epsilon,(k \kern-1pt - \kern-1pt 1)n)$ and distance $\delta$.

We continue to use the notation $\delta = \vert \gamma_2 - \gamma_1 \vert$ for the distance of $k$-path pairs.  For any $(\gamma_1,\gamma_2) \in \mathcal{P}_{n,\delta}^{k,\epsilon}$, it is always the case that $1 \leq \delta \leq n$, with the maximum distance of $n$ only being obtained by the pair with $\gamma_1 = N^{n-\epsilon}$ and $\gamma_2 = (E N^{k-2})^n$.  It follows that the sets $\mathcal{P}_{n,\delta}^{k,\epsilon}$ encompass all nonempty collections of $k$-path pairs if we range over $1 \leq \delta \leq n$ and $0 \leq \epsilon \leq (k-1)n$.

\subsection{Generalized $k$-Path Pairs with $0 \leq \epsilon \leq k-1$}
\label{subsec: k-path pairs, case 1}

In order to enumerate arbitrary $\mathcal{P}_{n,\delta}^{k,\epsilon}$, we fix $k,\epsilon$ and define a recurrence with respect to $n,\delta$.  This recurrence will directly generalize Shapiro's original recurrence for the Catalan triangle \cite{Shapiro}.  We begin with the range $0 \leq \epsilon \leq k-1$, where the recursion will eventually correspond to the $Z$- and $A$-sequences of a proper Riordan array.

\begin{theorem}
\label{thm: k-path pair recursion}
For any $k \geq 2$, $n \geq 1$, and $0 \leq \epsilon \leq k-1$,
$$\vert \mathcal{P}_{n,\delta}^{k,\epsilon} \vert = 
\begin{cases}
\displaystyle{\sum_{j=1}^k \binom{k}{j} \vert \mathcal{P}_{n-1,j}^{k,\epsilon} \vert - \sum_{j=1}^\epsilon \binom{\epsilon}{j} \vert \mathcal{P}_{n-1,j}^{k,\epsilon} \vert} & \text{for } \delta = 1, \text{and} \\[2.5ex]
\displaystyle{\sum_{j=0}^k \binom{k}{j} \vert \mathcal{P}_{n-1,\delta-1+j}^{k,\epsilon} \vert} & \text{for } \delta > 1.
\end{cases}$$
\end{theorem}
\begin{proof}

For any length-$(k - 1)$ word $w$ in the alphabet $\lbrace E,N \rbrace$, define $U_w$ to be the set of all $(\gamma_1,\gamma_2) \in \mathcal{P}_{n,\delta}^{k,\epsilon}$ such that $\gamma_1$ terminates with $w$ and $\gamma_2$ terminates with $N^{k-1}$.  If $w$ contains precisely $j$ instances of $E$, this implies $\gamma_1 = \eta_1 w$ and $\gamma_2 = \eta_2 N^{k-1}$ for some $(\eta_1,\eta_2) \in \mathcal{P}_{n-1,\delta+j}^{k,\epsilon}$.  Similarly define $V_w$ to be all $(\gamma_1,\gamma_2) \in \mathcal{P}_{n,\delta}^{k,\epsilon}$ such that $\gamma_1$ terminates with $w$ and $\gamma_2$ terminates with $E N^{k-2}$.  If $w$ contains $j$ precisely instances of $E$, then $\gamma_1 = \eta_1 w$ and $\gamma_2 = \eta_2 E N^{k-2}$ for some $k$-path pair $(\eta_1,\eta_2) \in \mathcal{P}_{n-1,\delta+j-1}^{k,\epsilon}$.  By construction, $\mathcal{P}_{n,\delta}^{k,\epsilon} = (\bigcup_w U_w) \cup (\bigcup_w V_w)$.

See Figure \ref{fig: recursion proof diagrams} for the general form of terminal subpaths in an element $(\gamma_1,\gamma_2)$ of $U_w$ or $V_w$.  In both diagrams, $(a,b)$ is fixed as the terminal point of $\gamma_1$, whereas the final $k-1$ steps of $\gamma_1$ are determined by $w$ and lie within the dotted triangle in the upper-left of each image.

Now take any length-$(k-1)$ word $w$ with precisely $j$ instances of $E$.  Our strategy is to enumerate $U_w$ and $V_w$ via consideration of the injective maps $g_w: \mathcal{P}_{n-1,\delta+j}^{k,\epsilon} \rightarrow S$, $g_w(\eta_1,\eta_2)=(\eta_1 w,\eta_2 N^{k-1})$ and $h_w: \mathcal{P}_{n-1,\delta+j-1}^{k,\epsilon} \rightarrow S$, $h_w(\eta_1,\eta_2)=(\eta_1 w,\eta_2 E N^{k-2})$.  Here $S$ denotes some collection of path-pairs whose elements may intersect apart from at $(0,0)$.  We clearly have $U_w \subseteq \im(g_w)$ and $V_w \subseteq \im(h_w)$ for any word $w$.  We also have $U_w = \im(g_w)$ iff every path pair in $\im(g_w)$ is non-intersecting apart from $(0,0)$, and $\im(h_w) = V_w$ iff every path pair in $\im (h_w)$ is non-intersecting apart from $(0,0)$.

Begin with $g_w$.  The path pair $g(\eta_1,\eta_2) = (\eta_1 w,\eta_2 N^{k-1})$ can only feature an intersection away from $(0,0)$ if the final $k-1$ steps of $\eta_1 w$ pass through some northwest corner of $\eta_2 N^{k-1}$.  As seen in Figure \ref{fig: recursion proof diagrams}, the largest possible $y$-coordinate for a northwest corner of $\eta_2 N^{k-1}$ is $b - \delta + \epsilon - 2k + 3$, whereas the terminal point of $\eta_1$ has a $y$-coordinate of at least $b - k + 1$.  Since we're assuming $\epsilon \leq k-1$, we have $\epsilon \leq \delta (k-1)$ for all $\delta \geq 1$. It follows that $b - \delta + \epsilon - 2k + 3 \leq b - k + 1$ for all $\delta \geq 1$, with the case of $b - d + \epsilon - 2k + 3 = b - k + 1$ being impossible because the input path $(\eta_1,\eta_2)$ was assumed to be non-intersecting away from $(0,0)$.  This implies that $\eta_1 w$ cannot intersect $\eta_2 N^{k-1}$ away from $(0,0)$ for any word $w$.

It follows that $g_w$ represents a bijection from $\mathcal{P}_{n-1,\delta+j}^{k,\epsilon}$ onto $U_w$ for every word $w$ when $\epsilon \leq k-1$.  Since there are $\binom{k-1}{j}$ words $w$ with precisely $j$ instances of $E$, a total of $\binom{k-1}{j}$ sets $U_w$ lie in bijection with $\mathcal{P}_{n-1,\delta+j}^{k,\epsilon}$ for each $0 \leq j \leq j-1$.  This gives

\begin{equation}
\label{eq: Uw summation}
\sum_w \vert U_w \vert = \sum_{j=0}^{k-1} \binom{k-1}{j} \vert \mathcal{P}_{n-1,\delta+j}^{k,\epsilon} \vert = \sum_{j=1}^{k} \binom{k-1}{j-1} \vert \mathcal{P}_{n-1,\delta+j-1}^{k,\epsilon} \vert.
\end{equation}

For $h_w$, we separately consider the cases of $\delta = 1$ and $\delta \geq 2$.  Begin by assuming $\delta \geq 2$.  We once again note that $h_w(\eta_1,\eta_2) = (\eta_1 w, \eta_2 EN^{k-2})$ has intersections away from $(0,0)$ only when the final $k-1$ steps of $\eta_1 w$ intersect some northwest corner of $\eta_2 EN^{k-2}$.  From Figure \ref{fig: recursion proof diagrams}, since $\delta \geq 2$ we see that the $y$-coordinate of such a corner can be at most $b - \delta + \epsilon - 2k + 4$.  Our assumptions of $\epsilon \leq k-1$ and $\delta \leq 2$ together ensure $\epsilon \leq k-3+\delta$ and thus that $b-\delta+\epsilon-2k+4 \leq b - k + 1$, with the case of $b-\delta+\epsilon-2k+4 = b - k + 1$ being impossible because we've assumed that $(\eta_1,\eta_2)$ lacks intersections away from $(0,0)$.  This implies that $\eta_1 w$ cannot intersect $\eta_2 E N^{k-2}$ away from $(0,0)$ for any word $w$ when $\delta \geq 2$, and thus that $h_w$ is a bijection from $\mathcal{P}_{n-1,\delta+j-1}^{k,\epsilon}$ onto $V_w$ for every word $w$ when $\delta \geq 2$.

When $\delta=1$, the map $h_w$ may introduce new intersections.  Fixing $w$, either every image $h_w(\eta_1,\eta_2) = (\eta_1 w, \eta_2 E N^{k-2})$ will have an intersection away from $(0,0)$, or every image $h_w(\eta_1,\eta_2)$ will lack such an intersection.  That first subcase implies that the corresponding set $V_w$ is empty, whereas that second subcase implies that $V_w$ is nonempty and in bijection with $\mathcal{P}_{n-1,\delta+j-1}^{k,\epsilon}$.  We only need to enumerate how many words $w$ (for each choice of $0 \leq j \leq k-1$) fall into each subcase.

As seen on the right side of Figure \ref{fig: recursion proof diagrams}, when $\delta=1$ the final northwest corner of $\eta_2 E N^{k-2}$ occurs at $(a,b+\epsilon-k+1)$.  Fixing a word $w$ with precisely $j$ instances of $E$, we also see that $\eta_1$ terminates at $(a-j,b-k+j+1)$.  This means that $\eta_1$ can only pass through $(a,b+\epsilon-k+1)$ if $j \leq \epsilon$.  For any such $j \leq \epsilon$, there are precisely $\binom{\epsilon}{j}$ words $w$ in which this additional intersection occurs.  As there are $\binom{k-1}{j}$ words $w$ with precisely $j$ instances of $E$, if $\epsilon \leq k-1$ we know that $V_w$ is nonempty for precisely $\binom{k-1}{j} - \binom{\epsilon}{j}$ choices of $w$.  Combining our results for $\delta \geq 2$ and $\delta =1$ gives

\begin{equation}
\label{eq: Vw summation}
\sum_w \vert V_w \vert = \begin{cases}
\displaystyle{\sum_{j=0}^{k-1} \binom{k-1}{j} \vert \mathcal{P}_{n-1,\delta+j-1}^{k,\epsilon} \vert} & \text{for } \delta \geq 2, \text{ and} \\[3.5ex]
\displaystyle{\sum_{j=0}^{k-1} \left( \binom{k-1}{j} - \binom{\epsilon}{j} \right) \vert \mathcal{P}_{n-1,\delta+j-1}^{k,\epsilon} \vert} & \text{for } \delta=1.
\end{cases}
\end{equation}

\noindent Once again noting that $\mathcal{P}_{n,\delta}^{k,\epsilon} = (\bigcup_w U_w) \cup (\bigcup_w V_w)$, for $\delta \geq 2$ we have
$$\vert \mathcal{P}_{n,\delta}^{k,\epsilon} \vert = \sum_w \vert U_w \vert + \sum_w \vert V_w \vert = \sum_{j=1}^{k} \binom{k-1}{j-1} \vert \mathcal{P}_{n-1,\delta+j-1}^{k,\epsilon} \vert + \sum_{j=0}^{k-1} \binom{k-1}{j} \vert \mathcal{P}_{n-1,\delta+j-1}^{k,\epsilon} \vert$$

$$= \sum_{j=0}^k \left( \binom{k-1}{j-1} + \binom{k-1}{j} \right) \vert \mathcal{P}_{n-1,\delta+j-1}^{k,\epsilon}\vert = \sum_{j=0}^k \binom{k}{j} \vert \mathcal{P}_{n-1,\delta+j-1}^{k,\epsilon} \vert.$$

\noindent For $\delta = 1$, the facts that $0 \leq \epsilon \leq k-1$ and $\vert \mathcal{P}_{n-1,0}^{k,\epsilon} \vert = 0$ prompt the similar result
$$\vert \mathcal{P}_{n,1}^{k,\epsilon} \vert = \sum_w \vert U_w \vert + \sum_w \vert V_w \vert = \sum_{j=1}^{k} \binom{k-1}{j-1} \vert \mathcal{P}_{n-1,j}^{k,\epsilon} \vert + \sum_{j=0}^{k-1} \left( \binom{k-1}{j} - \binom{\epsilon}{j} \right) \vert \mathcal{P}_{n-1,j}^{k,\epsilon} \vert$$
$$= \sum_{j=0}^{k} \left( \binom{k-1}{j-1} - \binom{k-1}{j} \right) \vert \mathcal{P}_{n-1,j}^{k,\epsilon} \vert - \sum_{j=0}^{k-1} \binom{\epsilon}{j} \vert \mathcal{P}_{n-1,j}^{k,\epsilon} \vert = \sum_{j=1}^{k} \binom{k}{j} \vert \mathcal{P}_{n-1,j}^{k,\epsilon} \vert - \sum_{j=1}^{\epsilon} \binom{\epsilon}{j} \vert \mathcal{P}_{n-1,j}^{k,\epsilon} \vert$$

\end{proof}

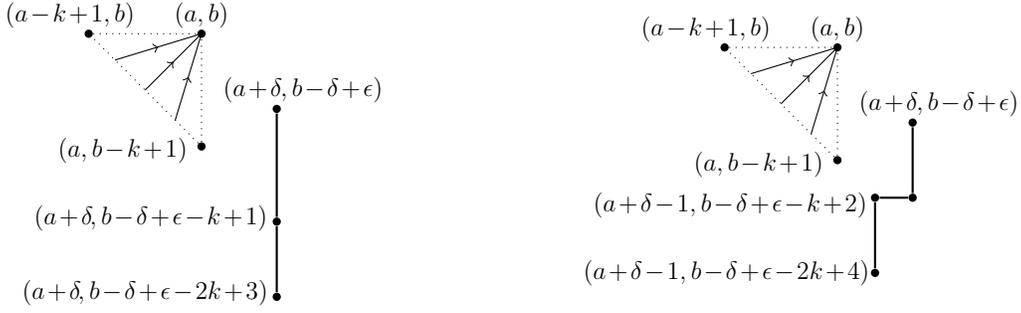
\begin{figure}[ht!]
\centering

\begin{tikzpicture}
[scale=.5,auto=left,every node/.style={circle, fill=black,inner sep=1.1pt}]
	\node (1) at (0,0) {};
	\node (1down) at (0,-3) {};
	\node (1left) at (-3,0) {};
	\draw[dotted] (1down) to (1left);
	\draw[dotted] (1down) to (1);
	\draw[dotted] (1left) to (1);
	\draw[->-] (-1.5,-1.5) to (0,0);
	\draw[->-] (-2.3,-0.7) to (0,0);	
	\draw[->-] (-0.7,-2.3) to (0,0);
	\node (2a) at (2,-2) {};
	\node (2b) at (2,-5) {};
	\node (2c) at (2,-7) {};
	\draw[thick] (2a) to (2b);
	\draw[thick] (2b) to (2c);
	\node[fill=none] (x1) at (0,.5) {\scalebox{.8}{$(a,b)$}};
	\node[fill=none] (x2) at (-2.1,-3.1) {\scalebox{.8}{$(a,b \kern-1.5pt - \kern-1.5pt k \kern-1.5pt + \kern-1.5pt 1)$}};
	\node[fill=none] (x3) at (-3.5,.5) {\scalebox{.8}{$(a \kern-1.5pt - \kern-1.5pt k \kern-1.5pt + \kern-1.5pt 1,b)$}};
	\node[fill=none] (x4) at (2.7,-1.5) {\scalebox{.8}{$(a \kern-1.5pt +\kern-1.5pt \delta,b \kern-1.5pt - \kern-1.5pt \delta \kern-1.5pt + \kern-1.5pt \epsilon)$}};
	\node[fill=none] (x4) at (-1.35,-4.9) {\scalebox{.8}{$(a \kern-1.5pt + \kern-1.5pt \delta \kern-1.5pt, b \kern-1.5pt - \kern-1.5pt \delta \kern-1.5pt + \kern-1.5pt \epsilon \kern-1.5pt - \kern-1.5pt k \kern-1.5pt + \kern-1.5pt 1)$}};
	\node[fill=none] (x4) at (-1.5,-6.9) {\scalebox{.8}{$(a \kern-1.5pt + \kern-1.5pt \delta \kern-1.5pt, b \kern-1.5pt - \kern-1.5pt \delta \kern-1.5pt + \kern-1.5pt \epsilon \kern-1.5pt - \kern-1.5pt 2k \kern-1.5pt + \kern-1.5pt 3)$}};
\end{tikzpicture}
\hspace{.9in}
\begin{tikzpicture}
[scale=.5,auto=left,every node/.style={circle, fill=black,inner sep=1.1pt}]
	\node (1) at (0,0) {};
	\node (1down) at (0,-3) {};
	\node (1left) at (-3,0) {};
	\draw[dotted] (1down) to (1left);
	\draw[dotted] (1down) to (1);
	\draw[dotted] (1left) to (1);
	\draw[->-] (-1.5,-1.5) to (0,0);
	\draw[->-] (-2.3,-0.7) to (0,0);	
	\draw[->-] (-0.7,-2.3) to (0,0);
	\node (2a) at (2,-2) {};
	\node (2b) at (2,-4) {};
	\node (2c) at (1,-4) {};
	\node (2d) at (1,-6) {};
	\draw[thick] (2a) to (2b);
	\draw[thick] (2b) to (2c);
	\draw[thick] (2c) to (2d);
	\node[fill=none] (x1) at (0,.5) {\scalebox{.8}{$(a,b)$}};
	\node[fill=none] (x2) at (-2.1,-3.1) {\scalebox{.8}{$(a,b \kern-1.5pt - \kern-1.5pt k \kern-1.5pt + \kern-1.5pt 1)$}};
	\node[fill=none] (x3) at (-3.5,.5) {\scalebox{.8}{$(a \kern-1.5pt - \kern-1.5pt k \kern-1.5pt + \kern-1.5pt 1,b)$}};
	\node[fill=none] (x4) at (2.7,-1.5) {\scalebox{.8}{$(a \kern-1.5pt +\kern-1.5pt \delta,b \kern-1.5pt - \kern-1.5pt \delta \kern-1.5pt + \kern-1.5pt \epsilon)$}};
	\node[fill=none] (x4) at (-2.85,-4.2) {\scalebox{.8}{$(a \kern-1.5pt + \kern-1.5pt \delta \kern-1.5pt - \kern-1.5pt 1, b \kern-1.5pt - \kern-1.5pt \delta \kern-1.5pt + \kern-1.5pt \epsilon \kern-1.5pt - \kern-1.5pt k \kern-1.5pt + \kern-1.5pt 2)$}};
	\node[fill=none] (x4) at (-2.95,-6) {\scalebox{.8}{$(a \kern-1.5pt + \kern-1.5pt \delta \kern-1.5pt - \kern-1.5pt 1, b \kern-1.5pt - \kern-1.5pt \delta \kern-1.5pt + \kern-1.5pt \epsilon \kern-1.5pt - \kern-1.5pt 2k \kern-1.5pt + \kern-1.5pt 4)$}};
\end{tikzpicture}

\vspace{-.4in}

\caption{Terminal subpaths for arbitrary $(\gamma_1,\gamma_2) \in U_w$ (left side) and arbitrary $(\gamma_1,\gamma_2) \in V_w$ (right side), as referenced in the proof of Theorem \ref{thm: k-path pair recursion}.}
\label{fig: recursion proof diagrams}
\end{figure}

It should be noted that the methods from Theorem \ref{thm: k-path pair recursion} may be extended to a somewhat broader range of parameters than $\epsilon \leq k-1$.  In particular, the summation of \eqref{eq: Uw summation} may be shown to hold for all $\epsilon \leq (k-1)\delta$, whereas the $\delta \geq 2$ summation of \eqref{eq: Vw summation} may be shown to hold for all $\epsilon \leq (k-1)(\delta-1)$.  Sadly, developing a general recursive relation for the full $\epsilon \leq \delta (k-1)$ range of Theorem \ref{thm: summation decomposition corollary 2} is extremely involved.  The enumerative usage of those recursions is also limited when $\epsilon > k-1$, as they no longer qualify as the $A$- and $Z$-sequences of a proper Riordan array.  As such, we delay the $\epsilon > k-1$ case until Subsection \ref{subsec: k-path pairs, case 2}, where generating function techniques may be applied to directly derive closed formulas from preexisting results for the general case.

For each choice of $k \geq 2$ and $0 \leq \epsilon \leq k-1$, the recursive relations of Theorem \ref{thm: k-path pair recursion} may be used to generate an infinite lower-triangular matrix $A^{k,\epsilon}$ whose $(i,j)$ entry is $a_{i,j}^{k,\epsilon} = \vert \mathcal{P}_{i+1,j+1}^{k,\epsilon} \vert$.  These $A^{k,\epsilon}$ qualify as proper Riordan arrays:

\begin{theorem}
\label{thm: k-path pairs as Riordan arrays}
For any $k \geq 2$ and $0 \leq \epsilon \leq k-1$, the integer triangle $A^{k,\epsilon}$ with $(i,j)$ entry $\vert \mathcal{P}_{i+1,j+1}^{k,\epsilon} \vert $ is the proper Riordan array $\mathcal{R}(C_k(t)^{k-\epsilon}, t C_k(t)^k)$, where $C_k(t)$ is the generating function for the $k$-Catalan numbers.
\end{theorem}
\begin{proof}
By Theorem \ref{thm: k-path pair recursion}, the array $A^{k,\epsilon}$ has $A$-sequence $A(t) = (1+t)^k$ and $Z$-sequence $Z(t) = \frac{(1+t)^k - (1+t)^\epsilon}{t}$.  The $k$-Catalan relation $C_k(t) = t C_k(t)^k + 1$ may then be used to verify the identities of \eqref{eq: Riordan arrays AZ to dh}:

$$t A(h(t)) = t (1 + tC_k(t)^k)^k = t C_k(t)^k = h(t),$$
$$\frac{d(0)}{1 + t Z(h(t))} = \frac{1}{1- t \frac{(1+tC_k(t)^k)^k - (1+tC_k(t)^k)^\epsilon}{tC_k(t)^k}} = \frac{1}{1-\frac{C_k(t)^k - C_k(t)^\epsilon}{C_k(t)^k}} = \frac{C_k(t)^k}{C_k(t)^\epsilon} = d(t).$$
\end{proof}

As they take the form $\mathcal{R}(C_k^i,C_k^j)$ for some $k \geq 2$ and some $i,j > 0$, every integer triangle $A^{k,\epsilon}$ is a Fuss-Catalan triangle of the type introduced by He and Shapiro \cite{HeShapiro}.  Many specific triangles $A^{k,\epsilon}$ also correspond to Riordan arrays that are well-represented in the literature.  The triangle $A^{2,0}$ is Shapiro's Catalan triangle, while $A^{2,0}$ and $A^{2,1}$ are two of the admissible matrices discussed by Aigner \cite{Aigner}.  More generally, whenever $\epsilon = 0$ the triangle $A^{k,\epsilon}$ is a renewal array with ``identical" $A$- and $Z$-sequences, as investigated by Cheon, Kim and Shapiro \cite{CKS}.

In a slight deviation from He and Shapiro \cite{HeShapiro}, we refer to $A^{k,\epsilon}$ as the \textbf{$\mathbf{(k,\epsilon)}$-Catalan triangle}.  See Figure \ref{fig: k-Shapiro triangles} for all $(k,\epsilon)$-Catalan triangles with $k=2,3,4$.

\begin{figure}[ht!]
\scalebox{.87}{
\setlength{\tabcolsep}{8pt}
\begin{tabular}{p{.43in} l l l l}
&\hspace{.2in} $\mathbf{\epsilon=0}$& \hspace{.2in} $\mathbf{\epsilon=1}$& \hspace{.2in} $\mathbf{\epsilon=2}$& \hspace{.2in} $\mathbf{\epsilon=3}$\\[6pt] 
$\mathbf{k=2}$
&
\setlength{\tabcolsep}{3.25pt}
\begin{tabular}{l l l l l}
1 & & & & \\
2 & 1 & & & \\
5 & 4 & 1 & & \\
14 & 14 & 6 & 1 & \\
42 & 48 & 27 & 8 & 1 \\
\end{tabular}
&
\setlength{\tabcolsep}{3.25pt}
\begin{tabular}{l l l l l}
1 & & & & \\
1 & 1 & & & \\
2 & 3 & 1 & & \\
5 & 9 & 5 & 1 & \\
14 & 28 & 20 & 7 & 1 \\
\end{tabular}
& & \\[44pt]
$\mathbf{k=3}$
&
\setlength{\tabcolsep}{3.25pt}
\begin{tabular}{l l l l l}
1 & & & & \\
3 & 1 & & & \\
12 & 6 & 1 & & \\
55 & 33 & 9 &  1 & \\
273 & 182 &  63 & 12 & 1 \\
\end{tabular}
&
\setlength{\tabcolsep}{3.25pt}
\begin{tabular}{l l l l l}
1 & & & & \\
2 & 1 & & & \\
7 & 5 & 1 & & \\
30 & 25 & 8 & 1 & \\
143 & 130 & 52 & 11 & 1 \\
\end{tabular}
&
\setlength{\tabcolsep}{3.25pt}
\begin{tabular}{l l l l l}
1 & & & & \\
1 & 1 & & & \\
3 & 4 & 1 & & \\
12 & 18 & 7 & 1 & \\
55 & 88 & 42 & 10 & 1 \\
\end{tabular}
& \\[44pt]
$\mathbf{k=4}$
&
\setlength{\tabcolsep}{3.25pt}
\begin{tabular}{l l l l l}
1 & & & & \\
4 & 1 & & & \\
22 & 8 & 1 & & \\
140 & 60 & 12 & 1 & \\
969 & 456 & 114 & 16 & 1 \\
\end{tabular}
&
\setlength{\tabcolsep}{3.25pt}
\begin{tabular}{l l l l l}
1 & & & & \\
3 & 1 & & & \\
15 & 7 & 1 & & \\
91 & 49 & 11 & 1 & \\
612 & 357 & 99 & 15 & 1 \\
\end{tabular}
&
\setlength{\tabcolsep}{3.25pt}
\begin{tabular}{l l l l l}
1 & & & & \\
2 & 1 & & & \\
9 & 6 & 1 & & \\
52 & 39 & 10 & 1 & \\
340 & 272 & 85 & 14 & 1 \\
\end{tabular}
&
\setlength{\tabcolsep}{3.25pt}
\begin{tabular}{l l l l l}
1 & & & & \\
1 & 1 & & & \\
4 & 5 & 1 & & \\
22 & 30 & 9 & 1 & \\
140 & 200 & 72 & 13 & 1 \\
\end{tabular}
\end{tabular}
}

\caption{Top five rows for all $(k,\epsilon)$-Catalan triangles $A^{k,\epsilon}$ with $k=2,3,4$.}
\label{fig: k-Shapiro triangles}
\end{figure}

One immediate consequence of Theorem \ref{thm: k-path pairs as Riordan arrays} is a closed formula for the size of every set $\mathcal{P}_{n,\delta}^{k,\epsilon}$ when $0 \leq \epsilon \leq k-1$.  Observe that every cardinality $\vert \mathcal{P}_{n,\delta}^{k,\epsilon} \vert = \frac{k\delta-\epsilon}{kn-\epsilon} \binom{kn-\epsilon}{n-\delta}$ from Corollary \ref{thm: k-path pairs as k-Catalan coefficients} is the Raney number $R_{k,k\delta-\epsilon}(n-\delta)$.  As defined by Hilton and Pedersen \cite{HP}, the Raney numbers (two-parameter Fuss-Catalan numbers) are defined to be $R_{k,r}(n) = [t^n]C_k(t)^r$, with the original $k$-Catalan numbers corresponding to $C_n^k = R_{k,1}(n) = R_{k,k}(n-1)$.

\begin{corollary}
\label{thm: k-path pairs as k-Catalan coefficients}
For any $k \geq 2$ and $0 \leq \epsilon \leq k-1$,
$$\vert \mathcal{P}_{n,\delta}^{k,\epsilon} \vert = [t^{n-\delta}]C_k(t)^{k\delta-\epsilon} = \frac{k\delta-\epsilon}{kn-\epsilon} \binom{kn-\epsilon}{n-\delta}$$
\end{corollary}
\begin{proof}
By the definition of $A^{k,\epsilon}$ we have $a_{i,j}^{k,\epsilon} = [t^i] C_k(t)^{k-\epsilon} (t C_k(t)^k)^j = [t^{i-j}] C_k(t)^{k-\epsilon + kj}$.  The corollary then follows from the fact that $\vert \mathcal{P}_{n,\delta}^{k,\epsilon} \vert = a_{n-1,\delta-1}^{k,\epsilon}$.
\end{proof}

\subsection{Generalized $k$-Path Pairs, All $\epsilon \geq 0$}
\label{subsec: k-path pairs, case 2}

If $\epsilon > k-1$, there need not be a bijection between $\mathcal{P}_{n,\delta}^{k,\epsilon}$ and some Raney number $R_{k,r}(n) = [t^n]C_k(t)^r$.  This implies that the cardinalities $\vert \mathcal{P}_{n,\delta}^{k,\epsilon} \vert$ cannot be organized into any Fuss-Catalan triangle.  One may still define an infinite lower-triangular array $A^{k,\epsilon}$ whose $(i,j)$ entry is $a_{i,j}^{k,\epsilon} = \vert \mathcal{P}_{i+1,j+1}^{k,\epsilon} \vert$, but for $\epsilon > k-1$ we always have $a_{0,0}^{k,\epsilon} = 0$ and the resulting arrays never qualify as a proper Riordan array.

For general $\epsilon$, we still have the following decomposition for $\vert \mathcal{P}_{n,\delta}^{k,\epsilon} \vert$:

\begin{proposition}
\label{thm: summation decomposition of k-path pairs}
Fix $n \geq 1$, $1 \leq \delta \leq n$, and $0 \leq \epsilon \leq (k-1)n$.  For any pair of non-negative integers $\epsilon_1,\epsilon_2$ such that $\epsilon = (k-1) \epsilon_1 + \epsilon_2$,

$$\vert \mathcal{P}_{n,\delta}^{k,\epsilon} \vert = \sum_{i=1}^{\delta} \binom{\epsilon_1}{\delta-i} \vert \mathcal{P}_{n-\epsilon_1,i}^{k,\epsilon_2} \vert.$$

\end{proposition}
\begin{proof}
As seen in Figure \ref{fig: path pair initial/final decomposition}, for any $(\gamma_1,\gamma_2) \in \mathcal{P}_{n,\delta}^{k,\epsilon}$ we may divide $\gamma_2$ into an initial subpath $\eta_1$ of length $n-(k-1)\epsilon_1$ and a terminal subpath $\eta_2$ of length $(k-1)\epsilon_1$.  As the length of $\eta_1$ is divisible by $k-1$, it is always the case that $(\gamma_1,\eta_1) \in \mathcal{P}_{n-\epsilon_1,i}^{k,\epsilon_2}$ for some $1 \leq i \leq \delta$.

Then consider the map $f: \mathcal{P}_{n,\delta}^{k,\epsilon} \rightarrow \bigcup_{i=1}^\delta \mathcal{P}_{n-\epsilon_1,i}^{k,\epsilon_2}$ where $f(\gamma_1,\gamma_2) = (\gamma_1,\eta_1)$.  This map is clearly surjective.  For any $1 \leq i \leq \delta$ and any $(\gamma_1,\eta_1) \in \mathcal{P}_{n-\epsilon_1,i}^{k,\epsilon_2}$, every way of appending precisely $\delta-i$ copies of $E^1 N^{k-2}$ and $\epsilon_1 - \delta + i$ copies of $N^{k-1}$ to the end of $\eta_1$ (in any order) produces an element of $\mathcal{P}_{n,\delta}^{k,\epsilon}$.  It follows that the inverse image $f^{-1}(\gamma_1',\gamma_2')$ of every $(\gamma_1',\gamma_2') \in \mathcal{P}_{n-\epsilon_1,i}^{k,\epsilon_2}$ has size $\binom{\epsilon_1}{\delta-i}$.  Ranging over $1 \leq i \leq \delta$ gives the required summation.
\end{proof}

\begin{figure}[ht!]
\centering
\begin{tikzpicture}
[scale=.5,auto=left,every node/.style={circle, fill=black,inner sep=1.1pt}]
	\draw[dotted] (0,0) to (5,0);
	\draw[dotted] (0,1) to (5,1);
	\draw[dotted] (0,2) to (5,2);
	\draw[dotted] (0,3) to (5,3);
	\draw[dotted] (0,4) to (5,4);
	\draw[dotted] (0,5) to (5,5);	
	\draw[dotted] (0,0) to (0,5);
	\draw[dotted] (1,0) to (1,5);
	\draw[dotted] (2,0) to (2,5);
	\draw[dotted] (3,0) to (3,5);
	\draw[dotted] (4,0) to (4,5);
	\draw[dotted] (5,0) to (5,5);
	\node (0*) at (0,0) {};
	\node (a1*) at (0,1) {};
	\node (a2*) at (0,2) {};
	\node (a3*) at (1,2) {};
	\node (a4*) at (1,3) {};
	\node (a5*) at (1,4) {};
	\node (b1*) at (1,0) {};
	\node (b2*) at (1,1) {};
	\node (b3*) at (2,1) {};
	\node (b4*) at (3,1) {};
	\node (b5*) at (4,1) {};
	\node (b6*) at (4,2) {};
	\node (b7*) at (5,2) {};
	\node (b8*) at (5,3) {};
	\node (b9*) at (5,4) {};
	\node (b10*) at (5,5) {};
	\draw[thick,dashed,color=red] (1,4) to (5,0);
	\draw[thick] (0*) to (a1*);	
	\draw[thick] (a1*) to (a2*);
	\draw[thick] (a2*) to (a3*);
	\draw[thick] (a3*) to (a4*);
	\draw[thick] (a4*) to (a5*);
	\draw[thick,color=blue] (0*) to (b1*);
	\draw[thick,color=blue] (b1*) to (b2*);
	\draw[thick,color=blue] (b2*) to (b3*);
	\draw[thick,color=blue] (b3*) to (b4*);
	\draw[thick,color=blue] (b4*) to (b5*);
	\draw[thick,color=green] (b5*) to (b6*);
	\draw[thick,color=green] (b6*) to (b7*);
	\draw[thick,color=green] (b7*) to (b8*);
	\draw[thick,color=green] (b8*) to (b9*);
	\draw[thick,color=green] (b9*) to (b10*);
\end{tikzpicture}
\caption{The decomposition of $\gamma_2$ for some $(\gamma_1,\gamma_2) \in \mathcal{P}_{10,4}^{2,5}$, as in the proof to Proposition \ref{thm: summation decomposition of k-path pairs}.  If $k > 2$, note that the initial subpath of $\gamma_2$ extends beyond the dotted diagonal line, until its length is divisible by $k-1$.} 
\label{fig: path pair initial/final decomposition}
\end{figure}
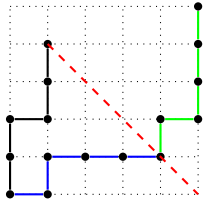

The summation on the right side of Proposition \ref{thm: summation decomposition of k-path pairs} may feature fewer than $\delta$ nonzero terms, as $\vert P_{n-\epsilon_1,i}^{k,\epsilon_2} \vert = 0$ when $n-\epsilon_1 < i$.  The decomposition $\epsilon = (k-1)\epsilon_1 + \epsilon_2$ also fails be be unique when $\epsilon \geq k-1$.  However, there always exists at least one decomposition of $\epsilon$ in which $\epsilon_2 \leq k-1$.

When $\epsilon \leq k-1$, this preferred decomposition of $\epsilon$ with $\epsilon_2 \leq k-1$ corresponds to $\epsilon_1 = 0$ and reduces the summation of Proposition \ref{thm: summation decomposition of k-path pairs} to the single term $\vert \mathcal{P}_{n,\delta}^{k,\epsilon} \vert$.  When $\epsilon > k-1$, choosing $\epsilon_1$ so that $\epsilon \leq k-1$ allows us to apply Corollary \ref{thm: k-path pairs as k-Catalan coefficients} to each term in the summation:

\begin{theorem}
\label{thm: summation decomposition corollary}
Fix $n \geq 1$, $1 \leq \delta \leq n$, and $0 \leq \epsilon \leq (k-1)n$.  For any pair of non-negative integers $\epsilon_1,\epsilon_2$ such that $\epsilon = (k-1)\epsilon_1 + \epsilon_2$ and $0 \leq \epsilon_2 \leq k-1$,

$$\vert \mathcal{P}_{n,\delta}^{k,\epsilon} \vert \ = \ [t^{n-\epsilon_1}] \kern+1pt \sum_{i=1}^\delta \binom{\epsilon_1}{\delta-i} t^i \kern+1pt C_k(t)^{ki-\epsilon_2} \ = \ \sum_{i=1}^\delta \frac{ki-\epsilon_2}{k(n-\epsilon_1)-\epsilon_2} \binom{\epsilon_1}{\delta-i} \binom{k(n-\epsilon_1) - \epsilon_2}{n-\epsilon_1 - i}.$$
\end{theorem}

Beyond the $\epsilon \leq k-1$ case of Subsection \ref{subsec: k-path pairs, case 1}, there are several situations where the general identity of Theorem \ref{thm: summation decomposition corollary} simplifies to give an enumeration equivalent to Corollary \ref{thm: k-path pairs as k-Catalan coefficients}.

\begin{theorem}
\label{thm: summation decomposition corollary 2}
Fix $n \geq 1$ and $0 \leq \epsilon \leq (k-1)n$, and take any pair of non-negative integers $\epsilon_1,\epsilon_2$ such that $\epsilon = (k-1)\epsilon_1 + \epsilon_2$ and $0 \leq \epsilon_2 \leq k-1$.  For all $\delta > \epsilon_1$, as well as for all $0 \leq \epsilon \leq (k-1)\delta$, we have

$$\vert \mathcal{P}_{n,\delta}^{k,\epsilon} \vert = [t^{n-\delta}] C_k(t)^{k\delta-\epsilon} = \frac{k \delta - \epsilon}{kn - \epsilon} \binom{kn - \epsilon}{n - \delta}.$$
\end{theorem}

\begin{proof}
Beginning with Theorem \ref{thm: summation decomposition corollary}, when $\delta - \epsilon_1 > 0$ we may rewrite the bounds of the summation and then perform the change of variables $j = \epsilon_1 - \delta + i$ to give
$$\vert \mathcal{P}_{n,\delta}^{k,\epsilon} \vert \ = \ [t^{n-\epsilon_1}] \sum_{i=1}^\delta \binom{\epsilon_1}{\delta-i} t^i \kern+1pt C_k(t)^{ki-\epsilon_2} \ = \ [t^{n-\epsilon_1}] \sum_{i=\delta - \epsilon_1}^\delta \binom{\epsilon_1}{\delta-i} t^i \kern+1pt C_k(t)^{ki-\epsilon_2}$$
$$= \ [t^{n-\epsilon_1}] \sum_{j=0}^{\epsilon_1} \binom{\epsilon_1}{j} t^{j+\delta-\epsilon_1} C_k(t)^{k(j+\delta-\epsilon_1)-\epsilon_2} \ = \ [t^{n-\epsilon_1}] t^{\delta - \epsilon_1} C_k(t)^{k \delta - k \epsilon_1 - \epsilon_2} \sum_{j=0}^{\epsilon_1} \binom{\epsilon_1}{j} (t \kern+1pt C_k(t)^k)^j$$

\noindent Recognizing the binomial expansion and applying the identity $C_k(t) = t \kern+1pt C_k(t)^k + 1$ yields
$$\vert \mathcal{P}_{n,\delta}^{k,\epsilon} \vert \ = \ [t^{n-\delta}] C_k(t)^{k\delta - k \epsilon_1 - \epsilon_2} (1 + t \kern+1pt C_k(t)^k)^{\epsilon_1} \ = \ [t^{n-\delta}] C_k(t)^{k\delta-k\epsilon_1 - \epsilon_2} C_k(t)^{\epsilon_1} = [t^{n-\delta}] C_k(t)^{k\delta - \epsilon}$$

\noindent For the second range of parameters given, we separately consider $\epsilon < (k-1)\delta$ and $\epsilon = (k-1)\delta$.  For the first subcase we always have $\epsilon < (k-1) \delta \leq (k-1) \delta + \epsilon_2$ and $\epsilon - \epsilon_2 = (k-1)\epsilon_1 < (k-1)\delta$, which implies $\epsilon_1 < \delta$ and allows us to apply our first result.  When $\epsilon = (k-1)\delta$ we may choose $\epsilon_1 = \delta-1$ and $\epsilon_2 = k-1$, which again implies $\epsilon_1 < \delta$.
\end{proof}

\section{Weak $k$-Path Pairs}
\label{sec: weak path pairs}

In this section, we loosen our restriction that generalized $k$-path pairs  $(\gamma_1,\gamma_2)$ cannot intersect apart from $(0,0)$ and merely require that $\gamma_1$ stays weakly above $\gamma_2$.  Formally, for any $k \geq 2$ and any set of non-negative integers $n,\epsilon,\delta$ such that $0 \leq \epsilon \leq (k-1)n$ and $0 \leq \delta \leq n$, we define $\widetilde{\mathcal{P}}_{n,\delta}^{k,\epsilon}$ to be the collection of ordered pairs $(\gamma_1,\gamma_2)$ of lattice paths that satisfy all of the following:

\begin{enumerate}
\item Both $\gamma_1$ and $\gamma_2$ begin at $(0,0)$ and use steps from $\lbrace E = (1,0), N = (0,1) \rbrace$,
\item $\gamma_2$ is composed of precisely $(k-1)n$ steps, the first of which is an $E$ step,
\item $\gamma_1$ is composed of precisely $(k-1)n-\epsilon$ steps, the first of which is an $N$ step,
\item $\gamma_1$ stays weakly above $\gamma_2$,
\item The difference between the terminal $x$ coordinates of $\gamma_1$ and $\gamma_2$ is $\delta$, and
\item $\gamma_2 = E^1 N^{b_1} E^1 N^{b_2} \hdots E^1 N^{b_m}$ satisfies $b_i = (k \kern-1pt - \kern-1pt 2) \kern-4pt \mod \kern-3pt(k \kern-1pt - \kern-1pt 1)$ for all $i$.
\end{enumerate}

We refer to any element $(\gamma_1, \gamma_2) \in \widetilde{\mathcal{P}}_{n,\delta}^{k,\epsilon}$ as a \textbf{weak $\mathbf{k}$-path pair of distance $\mathbf{\delta}$}.  Notice that $\delta=0$ is now possible when we also have $\epsilon=0$, corresponding to the case where $\gamma_1$ and $\gamma_2$ terminate at the same point.  We refer to this special case of $\delta=\epsilon=0$ as a \textbf{closed (weak) $\textbf{k}$-path pair}.  All nonempty sets $\widetilde{\mathcal{P}}_{n,\delta}^{k,\epsilon}$ fall within the ranges $0 \leq \delta \leq n$ and $0 \leq \epsilon \leq (k-1)n$.

Elements of $(\gamma_1, \gamma_2) \in \widetilde{\mathcal{P}}_{n,\delta}^{k,\epsilon}$ may then be subdivided according to the number of intersections between $\gamma_1$ and $\gamma_2$.  We let $\widetilde{\mathcal{P}}_{n,\delta,m}^{k,\epsilon}$ denote the collection of $(\gamma_1, \gamma_2) \in \widetilde{\mathcal{P}}_{n,\delta}^{k,\epsilon}$ where $\gamma_1$ and $\gamma_2$ intersect precisely $m$ times away from $(0,0)$, and we define such path pairs to be weak $k$-path pairs with $m$ \textbf{returns}.  It is easy to show that $\widetilde{\mathcal{P}}_{n,\delta,m}^{k,\epsilon}$ is empty unless $0 \leq m \leq n$, and that $\epsilon$ places further restrictions on which $m$ are possible.  For example, $m=n$ is only possible when $\epsilon=0$.

We henceforth call a closed $k$-path pair with only $m=1$ return as an \textbf{irreducible (closed) $\mathbf{k}$-path pair}.  Any weak $k$-path pair $(\gamma_1,\gamma_2) \in \widetilde{\mathcal{P}}_{n,\delta,m}^{k,\epsilon}$ with precisely $m$ returns may be uniquely decomposed into a sequence of subpath pairs $(\gamma_{1,1},\gamma_{2,1}),\hdots,(\gamma_{1,{m+1}},\gamma_{2,{m+1}})$ such that $(\gamma_{1,i},\gamma_{2,i})$ corresponds to an irreducible $k$-path pair for each $1 \leq i \leq m$ (after translating each subpath pair so that it begins at the origin).  If $(\gamma_1,\gamma_2)$ is a closed $k$-path pair, then the final subpath pair $(\gamma_{1,m+1},\gamma_{2,m+1})$ is empty.  Otherwise, that final subpath pair corresponds to some $k$-path pair $(\gamma'_1,\gamma'_2) \in \mathcal{P}_{n',\delta}^{k,\epsilon}$ for some $n' > 0$.

To enumerate $\widetilde{\mathcal{P}}_{n,\delta}^{k,\epsilon}$ and the $\widetilde{\mathcal{P}}_{n,\delta,m}^{k,\epsilon}$, we begin by enumerating irreducible $k$-path pairs:

\begin{proposition}
\label{thm: irreducible path pair enumeration}
Fix $k \geq 2$.  For any $n \geq 1$,
$$\vert \widetilde{\mathcal{P}}_{n,0,1}^{k,0} \vert = [t^{n-1}] C_k(t)^{k-1} = \frac{k-1}{kn-1} \binom{kn-1}{n-1}.$$
\end{proposition}
\begin{proof}
For any $(\gamma_1,\gamma_2) \in \widetilde{\mathcal{P}}_{n,0,1}^{k,0}$, observe that the final step of $\gamma_1$ must be an $E$ step.  This means that $\widetilde{\mathcal{P}}_{n,0,1}^{k,0}$ lies in bijection with $\mathcal{P}_{n,1}^{k,1}$, via the map the deletes the final step of $\gamma_1$.  The result then follows from Corollary \ref{thm: k-path pairs as k-Catalan coefficients}.
\end{proof}

Observe that $\widetilde{\mathcal{P}}_{n,0,1}^{2,0}$ is equivalent to the original notion of parallelogram polynominoes with semiperimeter $n$.  Proposition \ref{thm: irreducible path pair enumeration} recovers this preexisting combinatorial interpretation of the Catalan numbers as $\vert \widetilde{\mathcal{P}}_{n,0,1}^{2,0} \vert = [t^{n-1}] C(t) = C_{n-1}$.  For any $k \geq 2$, one could define the elements of $\widetilde{\mathcal{P}}_{n,0,1}^{k,0}$ as $k$-parallelogram polyominoes with semiperimeter $(k-1)n$, although for $k > 2$ these objects do not provide a combinatorial interpretation for the $k$-Catalan numbers.

The primary application of Proposition \ref{thm: irreducible path pair enumeration} is that it may be used to quickly enumerate any collection $\widetilde{\mathcal{P}}_{n,\delta,m}^{k,\epsilon}$, assuming $\epsilon$ and $\delta$ fall within the range proscribed by Theorem \ref{thm: summation decomposition corollary 2}:

\begin{theorem}
\label{thm: weak path pair enumeration}
Fix $n \geq 1$ and $k \geq 2$.  For any non-negative integers $\delta,\epsilon,m$ such that $\epsilon=\delta=0$ or $0 \leq \epsilon \leq (k-1)\delta$,

$$\vert \widetilde{\mathcal{P}}_{n,\delta,m}^{k,\epsilon} \vert \ = \ [t^{n-\delta - m}] C_k(t)^{k \delta - \epsilon + (k-1)m} \ = \ \frac{k \delta - \epsilon + (k-1)m}{kn- \epsilon - m} \binom{kn - \epsilon - m}{n - m - \delta}.$$
\end{theorem}
\begin{proof}
By Proposition \ref{thm: irreducible path pair enumeration}, for any $k \geq 2$ the generating function of irreducible $k$-path pairs is $\sum_{i=0}^\infty \vert \widetilde{\mathcal{P}}_{n,0,1}^{k,0} \vert \kern+1pt t^i = t \kern+1pt C_k(t)^{k-1}$. From Theorem \ref{thm: summation decomposition corollary 2}, when $0 \leq \epsilon < (k-1)\delta$ we also have the generating function $\sum_{i=0}^\infty \vert \mathcal{P}_{n,\delta}^{k,\epsilon} \vert t^i = t^\delta C_k(t)^{k\delta-\epsilon}$.  We treat the two cases of the theorem statement separately.

For the $\epsilon = \delta = 0$ case, every element of $\widetilde{\mathcal{P}}_{n,0,m}^{k,0}$ may be uniquely decomposed into a sequence of $m$ non-empty irreducible $k$-path pairs.  It follows that

$$\sum_{i=0}^\infty \vert \widetilde{\mathcal{P}}_{i,0,m}^{k,0} \vert \kern+1pt t^i = (t \kern+1pt C_k(t)^{k-1})^m = t^m C_k(t)^{(k-1)m}.$$

\noindent In this case we then have

$$\vert \widetilde{\mathcal{P}}_{n,0,m}^{k,0} \vert = [t^n] t^m C_k(t)^{(k-1)m} = [t^{n-m}] C_k(t)^{(k-1)m}.$$

For the $0 \leq \epsilon < (k-1)\delta$ case, every element of $\widetilde{\mathcal{P}}_{n,\epsilon,m}^{k,\delta}$ may be uniquely decomposed into a sequence of $m$ non-empty irreducible $k$-path pairs and an element of $\mathcal{P}_{n',\delta}^{k,\epsilon}$ for some $0 < n' < n - m$.  Here we have

$$\sum_{i=0}^\infty \vert \widetilde{\mathcal{P}}_{i,\epsilon,m}^{k,\delta} \vert \kern+1pt t^i = (t \kern+1pt C_k(t)^{k-1})^m \kern+2pt t^\delta C_k(t)^{k\delta-\epsilon} = t^{\delta+m} C_k(t)^{k\delta - \epsilon + (k-1)m}.$$

\noindent For this second case we then have

$$\vert \widetilde{\mathcal{P}}_{n,\epsilon,m}^{k,\delta} \vert = [t^n] t^{\delta+m} C_k(t)^{k\delta - \epsilon + (k-1)m} = [t^{n-\delta-m}] C_k(t)^{k \delta - \epsilon + (k-1)m}.$$
\end{proof}


\begin{thebibliography}{10}

\bibitem{Aigner} M. Aigner, Catalan-like numbers and determinants, \textit{J. Combin. Theory Ser. A} \textbf{87} (1999), 33--51.

\bibitem{Aigner2} M. Aigner, Enumeration via ballot numbers, \textit{Discrete Math.} \textbf{308} (2008), 2544--2563.

\bibitem{CKS} G.-S. Cheon, H. Kim and L.W. Shapiro, Combinatorics of Riordan arrays with identical A and Z sequences, \textit{Discrete Math.} \textbf{312(12-13)} (2012), 2040--2049.

\bibitem{DS} E. Deutsch and L.W. Shapiro, A Survey of the Fine numbers, \textit{Discrete Math.} \textbf{241} (2001), 241--265.

\bibitem{HeShapiro} T.-X. He and L.W. Shapiro, Fuss-Catalan matrices, their weighted sums, and stabilizer subgroups of the Riordan group, \textit{Linear Algebra Appl.} \textbf{532} (2017), 25--42.

\bibitem{HLM} S. Heubach, N.Y. Li and T. Mansour, Staircase tilings and $k$-Catalan structures, \textit{Discrete Math.} \textbf{308} (2008), no. 24, 5954--5964.

\bibitem{HP} P. Hilton and J. Pedersen, Catalan numbers, their generalizations, and their uses, \textit{Math. Intelligencer} \textbf{13} (1991), no. 2, 64--75.

\bibitem{MRSV} D. Merlini, D.G. Rogers, R. Sprugnoli and M.C. Verri, On some alternative characterizations of Riordan arrays, \textit{Canadian Jour. of Math.} \textbf{49(2)} (1997), 301--320.

\bibitem{Rogers} D.G. Rogers, Pascal triangles, Catalan numbers, and renewal arrays, \textit{Discrete Math.} \textbf{22} (1978), 301--310.

\bibitem{Shapiro} L.W. Shapiro, A Catalan triangle, \textit{Discrete Math.} \textbf{14} (1976), 83--90.

\bibitem{Stanley} R.P. Stanley, \textit{Catalan Numbers}, Cambridge University Press, 2015.

\end{thebibliography}
\end{document}